\documentclass[12pt,final]{amsart}
\newtheorem{definition}{Definition}[section]
\newtheorem{theorem}{Theorem}[section]

\newtheorem{algorithm}[theorem]{Algorithm}

\newtheorem{corollary}[theorem]{Corollary}

\newtheorem{lemma}[theorem]{Lemma}

\newtheorem{proposition}[theorem]{Proposition}

\usepackage{url,upref,amscd,amsmath,amsthm}
\usepackage{times}
\usepackage{lmodern} 
\usepackage[T1]{fontenc}
\usepackage[psamsfonts]{amssymb}
\usepackage{cite}
\usepackage{fullpage}
\usepackage{draftwatermark}
\SetWatermarkScale{1.5}
\SetWatermarkLightness{0.85}
\begin{document}
\title[MOR cryptosystem]{The MOR cryptosystem and finite $p$-groups}
\author[Mahalanobis]{Ayan Mahalanobis}
\address{IISER Pune, Dr.~Homi Bhabha Road, Pashan Pune-411008, India}
\date{\today}
\email{ayan.mahalanobis@gmail.com}
\thanks{This research was supported by a NBHM research grant}
\begin{abstract} 
The ElGamal cryptosystem is the most widely used public key cryptosystem. It uses the discrete logarithm problem as the cryptographic primitive. The MOR cryptosystem is a similar cryptosystem. It uses the discrete logarithm problem in the automorphism group as the cryptographic primitive. In this paper, we study the MOR cryptosystem for finite $p$-groups. The study is complete for $p^\prime$-automorphisms. For $p$-automorphisms there are some interesting open problems.
\end{abstract}
\maketitle
\section[Introduction]{Introduction}
This is a study of the MOR cryptosystem using finite $p$-groups. Similar studies were done by this author ~\cite{ayan3,ayan4}. The MOR cryptosystem, that we are going to describe in details shortly, works with the automorphism group of a group. As a matter of fact, we do not even need a group. Any finitely presented structure on which automorphisms can be defined will do. We can define the MOR cryptosystem on that structure. However, a MOR cryptosystem might not be secure or implementation-friendly. So this paper can be seen as a search for favorable groups for the MOR cryptosystem in the class of finite $p$-groups.

Once we decide that we will look into the class of $p$-groups, an obvious question surfaces. Are there $p$-groups on which the cryptosystem is secure? Once the answer is yes, then is it any better than the existing one? So we have three questions is front of us: 
\begin{description}
\item[1] Are there favorable $p$-groups? 
\item[2] Is the cryptosystem secure\footnote{There are many different definitions of security, we use the basic one -- find $m$, from the automorphism $\phi$ and its power $\phi^m$.} on those groups? 
\item[3] Is the cryptosystem faster on those groups compared to a suitably defined ElGamal cryptosystem?
\end{description}
To answer these questions, we had to divide the automorphisms in two different classes. One, $p$-automorphisms and the other $p^\prime$-automorphisms. For $p^\prime$-automorphisms we show that there are secure MOR cryptosystems on a $p$-group. However, they offer no advantage than working with matrices over the finite field $\mathbb{F}_p$. So, after reading this paper, one might argue and rightfully so: instead of using $p^\prime$-automorphisms and $p$-group, why not just use matrices of the right size? 

The case for $p$-automorphisms is little complicated and we say upfront that we have not been able to analyze it completely. The case of $p$-automorphisms break down into two sub-cases and we were able to deal with one easily. The other case is very interesting and we were able to shed some light into that with an example. However, a detailed analysis is missing and we leave it as ongoing research. The situation with $p$-automorphisms compared to $p^\prime$-automorphisms is in many ways similar to the modular representation theory compared to the ordinary representation theory. The later is much easier to deal with than the former.

\section{Definitions and Notations}
Most of the definitions used in this paper are standard and in Gorenstein~\cite{gorenstein}. However, we define a few of them for the convenience of the reader. All groups in this paper are finite. We use $p$ for a prime and $q$ for a prime-power.
\begin{definition}[$p^\prime$-automorphisms and $p$-automorphisms]
An automorphism $\phi$ of a $p$-group $G$ is a $p$-automorphism if its order is power of $p$ and $p^\prime$-automorphism if its order is coprime to $p$.
\end{definition}
In general, it is not true that an automorphism is either a $p$-automorphism or a $p^\prime$-automorphism. However, for the purpose of understanding the security of a MOR cryptosystem, due to the Pohlig-Hellman algorithm ~\cite[Section 2.9]{silverbook}, an automorphism is either a $p$-automorphism or a $p^\prime$-automorphism.
\begin{definition}[Special $p$-group]
Usually, a special $p$-group is defined to be a $p$-group such that $\mathcal{Z}(G)=G^\prime=\Phi(G)$ and is elementary-abelian. Here $G^\prime$, $\mathcal{Z}(G)$ and $\Phi(G)$ are the commutator subgroup, the center and the Frattini subgroup respectively. However, it is not very hard to show that the elemetary-abelian part is redundant.
\end{definition}
\begin{definition}[Favorable $p$-group]
A $p$-group $G$ is called a favorable $p$-group, if there is a non-identity $p^\prime$-automorphism $\phi$ of the group, such that, if the automorphism fixes a proper subgroup $H$ of $G$, it is the identity on $H$. 
\end{definition}
A good example of a favorable $p$-group is the elementary-abelian $p$-group, denoted by $G$. Any automorphism of that can be realized as a matrix. If the characteristic polynomial of an automorphism $\phi$ is irreducible, then there are no $\phi$-invariant proper subgroups of $G$. So the above condition is true vacuously.

A curious reader might find the requirement ``$p^\prime$-automorphism $\phi$'' unnecessary. The reason for the restriction is, for $p^\prime$-automorphism favorable $p$-groups is the right notion to look at. If there is a subgroup that is fixed by $\phi$, one can study the discrete logarithm problem on the action of the automorphism on that subgroup, unless the automorphism is the identity on that subgroup. We will see, in the case of $p^\prime$-automorphisms, the discrete logarithm problem in the automorphism group translates to the discrete logarithm problem in non-singular matrices. In the case of $p$-automorphisms, it is not clear if the notion of favorable $p$-group is the best way to go. We simply don't have enough examples of secure MOR cryptosystem using $p$-automorphisms of $p$-groups yet. So we refrain ourselves from defining favorable $p$-groups for $p$-automorphisms.
\section{The MOR cryptosystem}
In this section, we provide a somewhat detailed description of a small but important portion of \emph{public key cryptography}. We start with a cryptographic primitive -- \emph{the discrete logarithm problem}. The standard reference for public key cryptography is Hoffstein et.~al.~\cite{silverbook}.
\begin{definition}[The discrete logarithm problem]
Let $G=\langle g\rangle$ be a finite cyclic group of prime order. We are given $g$ and $g^m$ for some $m\in\mathbb{N}$. The discrete logarithm problem is to find the smallest $m$.  
\end{definition}
The discrete logarithm problem is neither secure or insecure. It being secure or insecure is a property of the presentation of the group. The property of being secure or insecure is not a group theoretic property, it is not invariant under isomorphism.

The discrete logarithm problem is the easiest in prime subgroups of $(\mathbb{Z}_n,+)$ and is considered secure in prime subgroups of the multiplicative group of a finite field $\mathbb{F}_q^\times$ and is considered really secure in a prime order subgroup of the rational points of an elliptic curve. The difference in security between finite fields and points on elliptic curve is, there is no known sub-exponential attack against the elliptic curves. 

A concept related to the discrete logarithm problem is the \textbf{Diffie-Hellman problem}. We have the same $G$ as before, this problem is: given $g$, $g^{m^\prime}$ and $g^{m^{\prime\prime}}$ compute $g^{m^\prime m^{\prime\prime}}$. It is clear that if we know how to solve the discrete logarithm problem, i.e., we can find $m^\prime$ (or $m^{\prime\prime}$), we can then solve the Diffie-Hellman problem. The reverse direction is not known.

The most popular and prolific public key cryptosystem is the \emph{ElGamal cryptosystem}. It works in any cyclic subgroup of a group $G$. However, it might not be secure in any group.
\subsection{Description of the ElGamal cryptosystem}
\begin{description}\label{keyex1}
\item[Private Key] $m$, $m\in\mathbb{N}$.
\item[Public Key] $g$ and $g^m$.
\end{description}
\paragraph{\textbf{Encryption}}
\begin{description}
\item[a] To send a message (plaintext) $a\in G$ Bob computes $g^r$
  and $g^{mr}$ for a random $r\in\mathbb{N}$.
\item[b] The ciphertext is $\left(g^r,g^{mr}a\right)$.
\end{description}
\paragraph{\textbf{Decryption}}
\begin{description}
\item[a] Alice knows $m$, so if she receives the ciphertext
  $\left(g^r,g^{mr}a\right)$, she computes $g^{mr}$ from $g^r$ and
  then $g^{-mr}$ and then computes $a$ from $g^{mr}a$.
\end{description}
It is known that the security of the ElGamal cryptosystem is equivalent to the Diffie-Hellman problem~\cite[Proposition 2.10]{silverbook}. A very similar idea is the MOR cryptosystem.
\subsection{Description of the MOR cryptosystem}\label{MOR}
Let $G=\langle g_1,g_2,\ldots,g_\tau\rangle$, $\tau\in\mathbb{N}$
be a finite group and $\phi$ a non-trivial automorphism of 
$G$. Alice's keys are as follows: 
\begin{description}\label{keyex}
\item[Private Key] $m$, $m\in\mathbb{N}$.
\item[Public Key] $\left\{\phi(g_i)\right\}_{i=1}^\tau$ and $\left\{\phi^m(g_i)\right\}_{i=1}^\tau$.
\end{description}
\paragraph{\textbf{Encryption}}
\begin{description}
\item[a] To send a message (plaintext) $a\in G$ Bob computes $\phi^r$
  and $\phi^{mr}$ for a random $r\in\mathbb{N}$.
\item[b] The ciphertext is $\left(\left\{\phi^r(g_i)\right\}_{i=1}^\tau,\phi^{mr}(a)\right)$.
\end{description}
\paragraph{\textbf{Decryption}}
\begin{description}
\item[a] Alice knows $m$, so if she receives the ciphertext
  $\left(\phi^r,\phi^{mr}(a)\right)$, she computes $\phi^{mr}$ from $\phi^r$ and
  then $\phi^{-mr}$ and then computes $a$ from $\phi^{mr}(a)$.
\end{description}
Alice knows the order of the automorphism $\phi$, she can use
the identity $\phi^{t-1}=\phi^{-1}$ whenever $\phi^t=1$ to compute
$\phi^{-mr}$.

It is easy to see the following: if one can solve the Diffie-Hellman problem in $\langle\phi\rangle$, he can break the MOR cryptosystem. This follows from the fact that $\phi^r$ and $\phi^m$ are public. If one can solve the Diffie-Hellman problem, one can compute $\phi^{mr}$ and get the plaintext. The next theorem proves the converse.
\begin{theorem}
If there is an oracle that can decrypt arbitrary ciphertext for the MOR cryptosystem, one can solve the Diffie-Hellman problem in $\langle\phi\rangle$.  
\end{theorem}
\begin{proof}
Assume that there is an oracle that can decrypt arbitrary MOR ciphertext. Now recall that 
$a=\phi^{-mr}\left(\phi^{mr}(a)\right)$. Now suppose we have $\phi^{m^\prime}$ and $\phi^{m^{\prime\prime}}$ 
and we want to compute $\phi^{m^\prime m^{\prime\prime}}$. Then tell the oracle that $\phi^{m^\prime}$ is 
the public key and 
$\left(\phi^{m^{\prime\prime}},g_i\right)$ is the ciphertext. The oracle will return $\phi^{-m^\prime m^{\prime\prime}}(g_i)$ as the plaintext. Once this game is played for $i=1,2,\ldots,\tau$. We know $\phi^{-m^\prime m^{\prime\prime}}(g_i)$ for $i=1,2,\ldots,\tau$ and hence $\phi^{m^\prime m^{\prime\prime}}$. Thus solving the Diffie-Hellman problem in $\langle\phi\rangle$.
\end{proof}
In this paper we are primarily interested in exploring finite
$p$-groups for the purpose of building a \emph{secure} MOR cryptosystem. As
is well known, security and computational efficiency goes hand in
hand.  So unless we explore the computational complexity of the MOR
cryptosystem, a security analysis is useless. So there are two
questions that we will explore in this paper:
\begin{description}
\item[a] Is it possible to build a secure MOR cryptosystem using finite
  $p$-groups?
\item[b] Does this MOR cryptosystem has any advantage over existing
  cryptosystems?
\end{description}
Before we answer these questions, we need to explain one aspect of the security of the discrete logarithm problem. It is easy to see, using the Chinese remainder theorem, that the discrete logarithm problem in any cyclic group can be reduced to a discrete logarithm problem in its Sylow subgroups. Then a discrete logarithm problem in the Sylow subgroup can be further reduced to the discrete logarithm problem in a group of prime power ~\cite[Section 2.9]{silverbook}. The end result is: the
security of the discrete logarithm problem in a group is the security of the discrete logarithm problem in the largest prime-order subgroup in that group. In practice, the group considered for an
efficient and secure implementation of the discrete logarithm problem
is a group of prime order\footnote{The reader must have noticed that in the definition of the discrete logarithm problem we used groups of prime order.}. 
From the above argument, it is clear that we should only study \textbf{automorphisms of prime order} for the MOR cryptosystem.
 
One way to study automorphisms of a finite $p$-group for the MOR cryptosystem is using linear
methods in nilpotent groups ~\cite[Chapter VIII]{huppert}. That is our principal objective in this paper. The idea is
to find a series of subgroups such that automorphisms act linearly
either on the subgroups or the quotients. We will soon
assume, if a subgroup is fixed under an automorphism then it is the identity on that subgroup. With this assumption, we only
have to look at the action of an automorphism on the sections of the series.

With these in mind, we look at the exponent-$p$ central series of a
finite $p$-group $G$. The series is defined as follows:
\[G=G_0\unrhd G_1\unrhd\ldots\unrhd G_k=1\]
where $G_{i+1}=\left[G,G_i\right]G_i^p$. This series is well known to
have elementary-abelian quotients and is used in many aspects of
computations with finite $p$-groups~\cite{newman}.

There are two possible orders of an automorphism of a $p$-group for the MOR cryptosystems:
\begin{description}
\item[i] The automorphism $\phi$ is of order $p$.
\item[ii] The order of $\phi$ is a prime different from $p$, i.e., a $p^\prime$-automorphism.
\end{description} 
This can again be subdivided into four different cases:
\begin{description}
\item[a] The automorphism is of order $p$ and is identity on all the
  quotients of the exponent-$p$ central series.
\item[b] The automorphism is of order $p$ and is not identity on at least
  one section of the exponent-$p$ central series.
\item[c] The automorphism is of order $p^\prime$ and is not identity on
  at least one section of the exponent-$p$ central series.
\item[d] The automorphism is of order $p^\prime$ and is identity on all
  sections of the exponent-$p$ central series.
\end{description}
Recall that $G_1$ is the Frattini subgroup $\Phi(G)$. A well known theorem of
Burnside says that:
\begin{theorem}[Burnside]
Let $\phi$ be an automorphism of a group $G$. If
$\gcd\left(\mathrm{o}(\phi),|\Phi(G)|\right)=1$ and $\phi$ induces the identity
automorphism on $G/\Phi(G)$, $\phi$ is the identity
automorphism on $G$.
\end{theorem}
\begin{proof}
For a proof see~\cite[Theorem 1.15]{russ} or~\cite[Theorem 5.1.4]{gorenstein}.
\end{proof}
This says, the case c above reduces to:
the automorphism $\phi$ is of order $p^\prime$ and is not identity on
$G/\Phi(G)$. In this case $\phi$ acts on $G/\Phi(G)$ linearly and the
discrete logarithm problem in $\phi$ deduces to the discrete logarithm
problem in matrices over $\mathbb{F}_p$. The size of the matrix is
the same as the cardinality of a set of minimal generators of the
$p$-group.

It is also well known, if d is the case then $\phi$ is the
identity~\cite[Theorem 5.3.2]{gorenstein}. So there is no point studying d.

So we have three cases to look at a, b and c above.  

It is well known that usually, the exception being groups of prime order,
$p$-groups come with lots of subgroups and normal subgroups. 
The most difficult issue that one faces in choosing a $p$-group and the automorphism
$\phi$ for the MOR cryptosystem is the presence of subgroups of the
$p$-group which is fixed by $\phi$. If this happens, 
the discrete logarithm problem in the automorphism $\phi$ is reduced
to the discrete logarithm problem in the restriction of $\phi$ to that
subgroup. This reduction is most undesirable. On the other hand, working with
non-abelian $p$-groups this reduction is bound to happen. For example, the
commutator and the center are non-trivial characteristic
subgroups. The way out of this situation is to ensure, if $\phi$
fixes any subgroup then it is the identity on that subgroup. Once this
condition is imposed, we will see that favorable groups with $p^\prime$-automorphism are reduced
to either the elementary abelian $p$-group or the class of $p$-groups
$G$ with $G^\prime=\mathcal{Z}(G)=\Phi(G)$ and $\Phi(G)$ is elementary
abelian. Here $G^\prime$ is the
commutator subgroup, $\mathcal{Z}(G)$ is the center and
$\Phi(G)$ is the Frattini subgroup of $G$. These two class of
groups together are known as the \textbf{special $p$-groups}. 
\section{MOR cryptosystems on finite $p$-groups using $p^\prime$-automorphisms}
In this section we look at the MOR cryptosystem over finite $p$-groups
with $p^\prime$-automorphisms. Our standard reference for group theory
is Gorenstein~\cite{gorenstein} and for linear algebra is
Roman~\cite{roman}. We start with a generalization of a celebrated theorem from the odd-order paper. 
\begin{theorem}
A solvable group $G$ possesses a characteristic subgroup $C$
with the following properties:
\begin{description}
\item Subgroup $C$ is nilpotent with nilpotency class less than or equal 2.
\item $\mathcal{Z}(C)$ is a maximal characteristic abelian subgroup of $G$.
\item $\mathcal{C}_G(C)=\mathcal{Z}(C)$.
\item Every nontrivial $p^\prime$-automorphism of $G$ induces a
  non-trivial automorphism on $C$.
\end{description} 
\end{theorem}
For a proof see ~\cite[Theorem 14.1]{russ}. The subgroup $C$
is called a \emph{Thompson critical subgroup}. We will refer to
it as a critical subgroup. There can be more than one critical
subgroup in a group. It is clear from the theorem above, in our
search for favorable $p$-groups, we should look at $p$-groups whose only critical subgroup is the 
whole group. We will call those groups \textbf{self-critical}. Since a
self-critical group is of class at most 2, we should look at $p$-groups of class at most $2$. Now 
if $p$ is odd, in a $p$-group of class 2, $(xy)^p=x^py^p$. This makes the subgroup $\Omega_1(G)$ of exponent 
$p$. Since $\Omega_1(G)$ is characteristic the following corollary follows immediately. 
\begin{corollary}
For an odd prime $p$, favorable $p$-groups are of class at most 2 with
exponent $p$.
\end{corollary}
Before we go any further we need to state a well known theorem due to Hall and 
Higman ~\cite[Theorem C]{hall-higman}. The proof is available in many standard 
textbooks ~\cite[Theorem 5.3.7]{gorenstein}, so we won't reproduce it. 
\begin{theorem}\label{thm1}
Let $G$ be a favorable $p$-group, then $G/G^\prime$ is elementary abelian.
\end{theorem}
To summarize, favorable $p$-groups are of class at most $2$ and $G/G^\prime$ is 
elementary-abelian. It follows that $G^\prime\leq\mathcal{Z}(G)$. Then both $G/G^\prime$ and 
$G/\mathcal{Z}(G)$ are elementary abelian $p$-groups. We also have a $p^\prime$-automorphism 
$\phi$, such that, if $\phi$ fixes a subgroup of $G$, it is the identity on that subgroup. In particular, $\phi$ is the identity on $G^\prime$ and $\mathcal{Z}(G)$. 
 
\noindent There are two different ways to look at this situation:

\paragraph{\textbf{Ordinary representation theory}} Let
$A=\langle\phi\rangle$ be the subgroup generated by $\phi$. Since
$\phi$ is a $p^\prime$-automorphism, the order of $A$ is coprime to
the order of the group $G$. We have a coprime action of $A$ on
$G$. In particular, we have a linear action of $A$ on
$V=G/G^\prime$. Since this action is coprime we have the celebrated
Maschke's  theorem~\cite[Theorem 3.3.1]{gorenstein} at our
disposal. The theorem states, if we have an $A$-invariant proper
subspace $W\subset V$, it has an $A$-invariant complement. In other words there
is an $A$-invariant subspace $W^\prime$ of $V$ such that $V=W\oplus W^\prime$.
\paragraph{\textbf{Linear algebra}} Another way to look at the same
situation is by linear algebra. Let $V=G/G^\prime$. Clearly $V$ is a
finite dimensional vector-space over $\mathbb{F}_p$. Corresponding to a linear transformation 
$\phi$ of $V$, we can define scalar multiplication such that $V$ is a finitely generated module 
over the principal ideal domain
$\mathbb{F}_p[x]$ ~\cite[Chapter 7]{roman}. 
We denote this module by $V_{\phi}$. The reason we are interested in this module $V_\phi$ is that 
the submodules of $V_\phi$ are the $\phi$-invariant subspaces of $V$. 
With this we have the full force of the theory of finitely generated modules over a principal 
ideal domain at our disposition; especially the decomposition theorem.

The minimal polynomial of $\phi$ is a generator of the annihilator ideal of $V_\phi$ in $\mathbb{F}_p[x]$. We denote it 
by $\mathfrak{m}_\phi$ and assume it to be monic. Let $\mathfrak{m}_\phi=f_1^{m_1}(x)f_2^{m_2}(x)\ldots f_{k}^{m_k}(x)$ be the decomposition of $\mathfrak{m}_\phi$ as product of irreducible monic polynomials. One can write $V_\phi=V_1\oplus V_2\oplus\ldots\oplus V_k$ where a generator of the annihilator ideal of each \emph{primary component} $V_i$ is $f_{i}^{m_i}$. Each $V_i$ can either be cyclic or can be broken down as direct sum of cyclic modules. This theory is very well-known and successful, so we will omit the details and ask any interested reader to consult a textbook in linear algebra -- Roman ~\cite{roman} being one of them.  
\begin{lemma}
Let $\phi$ be a non-identity $p^\prime$-automorphism on $V$, where $V$
is a finite-dimensional vector space over $\mathbb{F}_p$; such that, if $\phi$ fixes a subspace of $V$ then it is the identity on that subspace. The following is true:
\begin{itemize}
\item[a.]The characteristic polynomial $\chi{_\phi}$ of $\phi$ is irreducible.
\item[b.]The module $V_{\phi}$ is simple.
\end{itemize} 
\end{lemma} 
\begin{proof}
Recall that $V_{\phi}$ is a finitely generated module over a principal ideal domain
$\mathbb{F}_p[x]$. Let $\mathfrak{m}_\phi$ be the minimal polynomial of $V_\phi$. Assume that
$\mathfrak{m}_{\phi}=f_1^{m_1}(x)f_2^{m_2}(x)\ldots f_{k}^{m_k}(x)$,
where each $f_i(x)$ is monic irreducible over $\mathbb{F}_p$ and each $m_i$
is a non-negative integer. Define the set \[V_i=\left\{v\in V_{\phi}\;
  :\; f_i^{m_i}(\phi)v=0\right\}\]
Then the fundamental theorem of finitely generated module over a
principal ideal theorem says that $V_{\phi}=V_1\oplus V_2\oplus\ldots\oplus V_k$. Now
assume for a moment that $k$ is greater than 1. Then we have $V_{\phi}$
as direct sum of non-trivial submodules. Recall that submodules of
$V_{\phi}$ are the $\phi$-invariant subspaces of $V$. Then we have that
$V$ is a direct sum of two $\phi$-invariant subspaces of $V$. So
$\phi$ acts like identity on both these subspaces and hence is the
identity on $V$. So this subspace decomposition is impossible, forcing
$k$ to be 1.

We have deduced that $\mathfrak{m}_{\phi}=f(x)^l$ where $f(x)$ is monic irreducible and $l$ is a 
positive integer. If $l$ is greater than 1, take the subspace
$V^\prime=\left\{v\in V_{\phi}\;:\; f^{l-1}(\phi)v=0\right\}$. Also construct the
subgroup $A=\langle\phi\rangle$. Since $\gcd(|A|,p)=1$, from Maschke's
theorem the subspace $V^\prime$ has a complement. This means that there is
another $A$-invariant subspace $V^{\prime\prime}$ such that
$V=V^\prime\oplus V^{\prime\prime}$. Then using an argument similar to
the one in last paragraph, we show that $l=1$ and the minimal polynomial $\mathfrak{m}_\phi$
is irreducible.

From the above discussion it follows clearly that the module $V_{\phi}$ is cyclic with irreducible 
minimal polynomial. Since a cyclic module with irreducible minimal polynomial is non-derogatory~\cite[Theorem 7.11]{roman}, we have the characteristic 
polynomial the same as the minimal polynomial.

The fact the module is simple, follows the fact that the minimal polynomial of any submodule 
will divide the minimal polynomial of the module and the minimal polynomial of the module is irreducible.   
\end{proof}
It is easy to prove a partial converse of the above lemma.
\begin{lemma}\label{u1}
Let $\phi$ be a linear transformation on the finite dimensional vector space over $\mathbb{F}_q$. 
If the characteristic polynomial $\chi_\phi$ is irreducible, the only $\phi$-invariant subspaces of
 $V$ are $0$ and $V$.
\end{lemma}
\begin{proof}
We will consider $V_\phi$ as a module over $\mathbb{F}_q[x]$. Since $\chi_\phi$ is irreducible it is
 also the minimal polynomial. Now if $S$ is a submodule of $V_\phi$, then its minimal polynomial 
will divide $\chi_\phi$. Since $\chi_\phi$ is irreducible we have a proof.
\end{proof}
This lemma is the most useful lemma in this whole paper. This paper is in search of favorable 
$p$-groups and the corresponding automorphism. One way, and probably the easiest way, is to look 
at the characteristic polynomial corresponding to an automorphism. If that characteristic 
polynomial is irreducible, we have our favorable $p$-group and the necessary automorphism. 
\begin{theorem}
A favorable $p$-group $G$ is a special $p$-group.
\end{theorem}
\begin{proof}
We already know that $G$ is of class at most $2$ and  $V=G/G^\prime$ is an elementary-abelian $p$-group. Let $\phi$ be a $p^\prime$-automorphism group automorphism, such that, if is fixes a proper subgroup of $G$, then it is the identity on that subgroup. Since $G^\prime$ is characteristic, $\phi$ is the identity on $G^\prime$. Consider the module $V_\phi$ over $\mathbb{F}_p[x]$ corresponding to $\phi$. Then from the lemma above we know that the characteristic polynomial $\chi_{\phi}$ is irreducible and $V_\phi$ is simple.

In any finite $p$-group, $G^\prime\subseteq\Phi(G)$ and from above $G^\prime\subseteq\mathcal{Z}(G)$. To show $G^\prime=\mathcal{Z}(G)$, notice that $V_\phi$ is a simple module over $\mathbb{F}_p[x]$ and all submodules are $\phi$-invariant subspaces. So $\mathcal{Z}(G)/G^\prime$ cannot be a nontrivial submodule. Similar is the case with $\Phi(G)$.

So if we assume that $G$ is not elementary-abelian, then $G^\prime=\mathcal{Z}(G)=\Phi(G)$. 
\end{proof}

At this point it is clear, to build a secure and optimal MOR cryptosystem with non-abelian $p$-groups one should
look at special $p$-groups and an automorphism $\phi$ such that $\phi$
is identity on all subgroup it fixes. In particular $\phi$ must
centralize $\Phi(G)$, so smaller the $\Phi(G)$ the better. So it is clear that we should look for groups
with $\Phi(G)$ as small as possible. We conclude that for a
non-abelian $p$-group ($p$ odd) and $p^\prime$-automorphisms the best group is
an extra-special $p$-group of prime exponent. For abelian $p$-groups, we should look only at elementary-abelian $p$-groups. For $p$ even, we still have the extra-special groups but we can use any exponent.

\section{The MOR cryptosystem and elementary abelian $p$-group}
As is well known, an elementary abelian $p$-group is a vector space over $\mathbb{F}_p$ the field of $p$ elements. So one way to look at MOR cryptosystems over an elementary abelian group is MOR cryptosystems over a vector space. If we fix a basis for the vector space, any linear transformation gives rise to a matrix. So the discrete logarithm problem in invertible linear transformations turns out to be the discrete logarithm problem over non-singular matrices. So we need to say a few things about that. Before we do that, we also need to remind our reader that security and speed goes hand in hand. One reason, the discrete logarithm problem in matrices was avoided in cryptography was the belief that matrix exponentiation is much more expensive. The security advantage we gain from the discrete logarithm problem in matrices does not outweigh the cost of matrix exponentiation. This view was put down by Menezes \& Wu~\cite{menezes1}. However with the recent advances in matrix exponentiation by Leedham-Green~\cite{green-brien}, the above argument is no longer valid. We get into the details of this argument in this section.

\subsection{Solving the discrete logarithm problem in non-singular matrices} Let $g$ and $g^m$ belongs to GL$(d,q)$, the discrete logarithm problem is to find $m$. This problem can be easy and hard. For uni-triangular matrices, i.e., matrices with one on the diagonal and arbitrary field element on the upper half and zero on the lower half, it is very easy. On the other hand, with matrices with irreducible characteristic polynomial, the discrete logarithm problem is hard.

Following is the work of Menezes \& Wu~\cite{menezes1}, which is the best known algorithm to solve the discrete logarithm problem in matrices. This algorithm is basically a reduction of the discrete logarithm problem in GL$(d,q)$ to a finite (possibly trivial) extension of $\mathbb{F}_q$. 
\subsection{The Menezes-Wu algorithm}
\begin{description}
\item Input: $g$ and $g^m$.
\item Output: $m$.
\item From $g$, compute the characteristic polynomial $\chi_g$ of $g$.
\item From $g^m$, compute the characteristic polynomial $\chi_{g^m}$ of $g^m$.
\end{description}
Let $\left\{\alpha_1,\alpha_2,\ldots,\alpha_d\right\}$ be the characteristic roots of $g$. This list might contain repeating entries. The characteristic roots lie in some finite (possibly trivial) extension of $\mathbb{F}_q$.
Let $\left\{\beta_1,\beta_2,\ldots,\beta_d\right\}$ be the characteristic roots of $g^m$. This list might contain repeating entries. The roots lie in some finite (possibly trivial) extension of $\mathbb{F}_q$.

Then $\left\{\beta_1,\beta_2,\ldots,\beta_d\right\}$ is  $\left\{\alpha_{i_1}^m,\alpha_{i_2}^m,\ldots,\alpha_{i_d}^m\right\}$, where $(i_1,i_2,\ldots,i_d)$ is $(1,2,\ldots,d)$ permuted. Note that there is no obvious way to order characteristic roots, but following Menezes and Wu, we will assume that this permutation is not going to offer much resistance in computing $m$. In other words, we assume that we can find $\alpha_i$ and $\beta_j$ such that $\alpha_i^m=\beta_j$. Once we have this one can solve for $m\bmod\text{o}(\alpha_i)$, where $\text{o}(\alpha_i)$ is the multiplicative order of $\alpha_i$. From, solving the required numbers of discrete logarithm problems in the suitable extensions and then applying the Chinese remainder theorem, one can solve the discrete logarithm problem in non-singular matrices. Note that the $\alpha_i$ and subsequently the $\beta_j$ will be in some extension field (possibly trivial) of $\mathbb{F}_q$. The largest extension possible is $\mathbb{F}_{q^d}$ and this happens when the \textbf{characteristic polynomial is irreducible}.

The most serious attack on the discrete logarithm problem in a finite field is the sub-exponential attack like the index-calculus attack. In this attack, if we are solving the discrete logarithm problem in $\mathbb{F}_{q^d}$, the time-complexity is
$\exp{\left((c+o(1))(\log{q}^d)^\frac{1}{3}(\log\log{q}^d)^\frac{2}{3}\right)}$, where $c$ is a constant, see~\cite{oliver} and~\cite[Section 4]{koblitz}. It is clear, larger the $d$ more secure is the discrete logarithm problem in matrices. So we can now safely conclude, to work with the discrete logarithm problem in matrices one should work with \textbf{matrices with irreducible characteristic polynomial}.
\subsection{Exponentiation in non-singular matrices} This section is a brief introduction to an amazing algorithm by Leedham-Green~\cite[Section 10]{green-brien} to compute $g^m$ for some $g\in\text{GL}(d,q)$. We only deal with the case where the characteristic polynomial $\chi_g$ of $g$ is irreducible. 
\begin{algorithm}[Leedham-Green]\hspace*{\fill} \\
Input: a matrix $g$ of size $d$ over a finite field $\mathbb{F}_q$ and a positive integer $m$.\\
Output: $g^m$\\
\vspace{-0.5cm}
\begin{description}
\item[] Find a matrix $P$ such that $B=P^{-1}gP$ is in the Frobenius normal form.
\item[] Determine the minimal polynomial $\mathfrak{m}(x)$ of $B$. Since the Smith normal form is sparse, it is easy to compute the minimal polynomial -- it takes $O(d^2)$ field multiplications.
\item[] Compute $t^m\bmod \mathfrak{m}(t)$ in $F[t]/\mathfrak{m}(t)$ as $\mathfrak{l}(t)$.
\item[] Compute $C=\mathfrak{l}(B)$
\item[] Return $PCP^{-1}$. 
\end{description}
\end{algorithm}
Notice that the objective of the above algorithm was to compute the power of an arbitrary matrix. In our case, for a MOR cryptosystem the matrix is not arbitrary, we can choose our matrix. So one can first choose an irreducible polynomial $\mathfrak{m}$ of degree $d$ over $\mathbb{F}_q$. Then choose $g$ to be the companion matrix for that polynomial $\mathfrak{m}$. Since the minimal polynomial divides the characteristic polynomial, the minimal polynomial is $\mathfrak{m}$ as well. So the first two steps and the last step in the above algorithm becomes redundant.

Once $\mathfrak{m}$ is irreducible in the above algorithm the quotient $\mathbb{F}[t]/\mathfrak{m}(t)$ is a field. So the third step is essentially an exponentiation in the field $\mathbb{F}_{q^d}$. So apart from computing the $C$ in the above algorithm, exponentiation of a matrix with irreducible characteristic polynomial is the same as exponentiation in the finite field $\mathbb{F}_{q^d}$.
 
The following is now clear: the discrete logarithm problem in GL$(d,q)$ is almost the same, both in terms of security and speed, to a discrete logarithm problem in $\mathbb{F}_{d^q}$. Note that this conclusion is remarkably different than that of Menezes \& Wu~\cite{menezes1}, where they write-off completely the discrete logarithm problem in matrices.  

Next we show that elementary-abelian $p$-groups are favorable $p$-groups. 
\begin{lemma}
Let $V$ be a vector space over $\mathbb{F}_p$. Let $\phi$ be a non-singular linear transformation on $V$. If $p|\text{o}(\phi)$, then $V$ has a proper $\phi$-invariant subspace.
\end{lemma}
\begin{proof}
Let $A=\langle\phi\rangle$. Then the given condition implies that $p||A|$. Considering the fact that any finite abelian group is the direct product of its Sylow subgroups, we see that one can write $\phi=\phi_p\phi_{p^\prime}$. Where $\phi_p$ and $\phi_{p^\prime}$ are $p$ and $p^\prime$ non-trivial automorphism respectively. From the fact that $(x^q-1)=(x-1)^q$ for any $p$-power $q$, we see that all the eigenvalues of $\phi_p$ are $1\in\mathbb{F}_p$. Let $\mathcal{E}$ be the eigenspace of $1$ in $V$. Clearly $\mathcal{E}$ is a proper subspace of $V$. Let $v\in\mathcal{E}$. Then $\phi_p\phi_{p^\prime}(v)=\phi_{p^\prime}\phi_p(v)$, which implies $\phi_p\phi_{p^\prime}(v)=\phi_{p^\prime}(v)$. This proves that $\phi_{p^\prime}(v)\in\mathcal{E}$. So $\mathcal{E}$ is a $\phi$-invariant proper subspace of $V$.  
\end{proof}
\begin{theorem}An elementary-abelian $p$-group is a favorable $p$-group.
\end{theorem}
\begin{proof}
An elementary abelian $p$-group $V$ is a vector space over $\mathbb{F}_p$. Then the automorphism group of $V$ is GL$(V)$. Let $\phi$ be an automorphism with irreducible characteristic polynomial. Then $\phi$ is a $p^\prime$-automorphism. 
Then Lemma ~\ref{u1} proves the rest. 
\end{proof}
\section{The extra-special $p$-groups and its automorphism group}
As we saw before, if we are dealing with $p^\prime$-automorphisms, there are only two interesting class of finite $p$-groups. One is the elementary abelian $p$-group and the other are the extra-special $p$-groups. The case for extra special $p$-groups is interesting, because it provides us with non-abelian $p$-groups which is presented in the \emph{power-commutator} form and provides us with a secure MOR cryptosystem, thus showing that abstract presentations can be useful.
As we will see, the security with $p^\prime$-automorphisms reduces to the discrete logarithm problem in non-singular matrices. This enables us to argue that working with $p^\prime$-automorphisms of a $p$-group, one has no advantage from working with matrices. However, the case with $p$-automorphisms is not quite settled yet. We will see, as an example with the central automorphisms of the extra-special p-groups that there are some potential with $p$-groups. The potential is the impossibility of the reduction to matrices, which killed the $p^\prime$-automorphisms.
\subsection{Extra-special $p$-groups}
It is well known that any special $p$-group is of exponent at most $p^2$. We saw earlier that for odd prime $p$ we can concentrate on groups of exponent $p$. So for an odd prime $p$ our principal interest is in the extra-special $p$-group of exponent $p$. Our principal reference is Gorenstein~\cite[Section 5.5]{gorenstein}. We briefly summarize few facts about the extra-special $p$ group of exponent $p$ denoted by $G$.
\begin{description}
\item[] The order of $G$ is $p^{2n+1}$ for some positive integer $n$. The cardinality of the minimal set of generators is $2n$ and we denote that set by $\{x_1,y_1,x_2,y_2,\ldots,x_n,y_n\}$. There is a relation $[x_i,y_i]=z$, where $\mathcal{Z}(G)=\langle z\rangle$ and $z^p=1$. Furthermore, $[x_i,x_j]=1$ and $[x_i,y_j]=1$ for $i\neq j$.
\item[] The group $G$ is the central product of $n$ copies of the group of order $p^3$ given by
$$\langle x,y,z\; | \; x^p=y^p=z^p=1, [x,z]=1,[y,z]=1,[x,y]=z\rangle.$$
\item[] In the group $G$, $G=G^\prime=\Phi(G)$ and is cyclic of order $p$.
\end{description}  
In a $p$-group, finding all automorphisms is often a very hard job. However, for an extra-special $p$-groups it is not that hard. The automorphisms were studied extensively by Winter ~\cite{winter}. The study of automorphisms of an extra-special $p$-group is not that hard because of a bilinear map $B:G/G^\prime\times G/G^\prime\rightarrow \mathbb{F}_p$. The map is defined as follows, let $\bar{x},\bar{y}\in G/G^\prime$, then $[x,y]=z^a$ for some integer $a$. Then $B(\bar{x},\bar{y})=\bar{a}$, where $\bar{a}=a\bmod p$. It is known that $B$ is an alternating, non-degenerate bilinear form on $G/G^\prime$.

We will not do a detailed presentation of the automorphisms of the extra-special $p$-group of prime exponent. An interested reader can find that in Winter ~\cite{winter}. However, to facilitate further discussion we have to describe them briefly.

Since an extra-special $p$ group is of class $2$, we have that $[x^n,y]=[x,y]^n$. Recall that the center $\mathcal{Z}(G)$ is of prime order and any automorphism of $\mathcal{Z}(G)$ can be lifted to an automorphism of $G$. So we have a complete description of the automorphisms of $G$, that are not identity on $\mathcal{Z}(G)$.

So now we have to concentrate on the automorphisms that fix $\mathcal{Z}(G)$. It was shown by Winter that an automorphism $\phi$ of $G$ is an automorphism of $G/\mathcal{Z}(G)$ if and only if it is the identity on $\mathcal{Z}(G)$.

It was further shown that for prime exponent, the automorphisms that fix $\mathcal{Z}(G)$ is the symplectic group Sp$(2n,p)$. Winter denotes this subgroup of the automorphism group by $H$ and has shown that it is a normal subgroup of the automorphism group.

To summarize, there are two kinds of automorphisms:
\begin{description}
\item[a] Automorphisms that are not the identity on the center $\mathcal{Z}(G)$ of $G$. Since, any automorphism of the center can be extended to an automorphism of the whole group, and the center is cyclic. We have a complete understanding of these automorphisms. They are uninteresting to our cause.
\item[b] One that are identity on the center. These automorphisms form a normal subgroup of the automorphism group of $G$. We will call them $H$.
\end{description}
For obvious reasons we are interested in b above. Let $\phi$ be an automorphism that centralizes the center. Winter has shown that, $\bar{\phi}:G/\mathcal{Z}(G)\rightarrow G/\mathcal{Z}(G)$ is an automorphism of $G/\mathcal{Z}(G)$ preserving the bilinear form $B$. We will abuse the notation a little bit and call the automorphism on the central quotient $\phi$ as well.

An interesting normal subgroup of $H$ is the group of inner automorphisms $I$. Using the fact that the commutator $G^\prime\subseteq\mathcal{Z}(G)$ and the identity $ab=ba[a,b]$ for any $a,b\in G$, it is clear that an inner automorphism is of the form
\begin{eqnarray*}
&x_i\mapsto x_iz^{d_i}&\\
&y_i\mapsto y_iz^{d_i^\prime}&\text{where}\;\; 0\leq d_i,d_i^\prime <p.
\end{eqnarray*}
From the fact, the group of the inner automorphisms $I$ is isomorphic to $G/\mathcal{Z}(G)$, it follows that there are $p^{2n}$ inner automorphisms. It also follow from a simple counting argument on all possible choices of $d_i$ and $d_i^\prime$. From our understanding of the inner automorphisms, the following proposition is clear:
\begin{proposition}\label{u2}
An automorphism $\phi$ of $G$ is an inner automorphism if and only if it is the identity on $\mathcal{Z}(G)$ and $G/\mathcal{Z}(G)$. The inner automorphisms commute and constitutes the group of central automorphisms.
\end{proposition}
It is known ~\cite[3E]{winter}, $H/I$ is isomorphic to Sp$(2n,p)$. Recall that $G/\mathcal{Z}(G)$ is a symplectic vector space over $\mathbb{F}_p$. We next show that the extra-special $p$-group of prime exponent is a favorable $p$-group.
\begin{theorem}
For a odd prime $p$, the extra-special $p$ group of exponent $p$ is a favorable $p$-group.
\end{theorem} 
\begin{proof}
Let $\phi\in\text{Sp}(2n,p)$, such that $\chi_\phi$ is irreducible. From the above discussion, we can consider $\phi$ to be an automorphism of $G$ that is the identity on $G^\prime$. According to Lemma 4.5, there are no proper $\phi$-invariant subspaces of $G/G^\prime$, and from Lemma 5.2 $\phi$ is a $p^\prime$-automorphism. Now assume that $H$ is a proper $\phi$-invariant subspace of $G$. Then consider $HG^\prime$. Notice that $G^\prime=\Phi(G)$ and furthermore $\Phi(G)$ is the set of non-generators of $G$. Then it follows that $HG^\prime$ is a proper subgroup and so $HG^\prime/G^\prime$ is a proper $\phi$-invariant subspace of $G/G^\prime$. Which implies that $HG^\prime\subseteq G^\prime$ and furthermore $H\subseteq G^\prime$. 
\end{proof}
\begin{corollary}
For a odd prime $p$, the extra-special $p$-group of exponent $p$ is self-critical.
\end{corollary}
\begin{proof}
Let $G$ denote the extra-special $p$-group of exponent $p$ and $C$ be a critical subgroup of $G$. Then the condition $\text{C}_G(C)=\mathcal{Z}(G)$ implies that $C$ is not contained in $\mathcal{Z}(G)$. From the above theorem $G$ is a favorable $p$-group. Then there is a corresponding automorphism $\phi$. Let $V=G/G^\prime$ and construct $V_\phi$ and it is known to be simple. Consider the subgroup $CG^\prime$.
Then $CG^\prime$ is either the whole group or the center $\mathcal{Z}(G)$. Since it can't be $\mathcal{Z}(G)$, it is the whole group. Now notice that $G^\prime=\Phi(G)$ and $\Phi(G)$ is the set of non-generators of $G$. It follows that if $CG^\prime=G$, $C=G$. So $G$ is self-critical.
\end{proof}
\subsection{The case when $p=2$} In this case a theorem of Winter ~\cite[Theorem 1(c)]{winter} comes in handy.
\begin{theorem} Let $P$ be an extra-special group of order $2^{2n+1}$. Subgroups $H$ and $I$ are as defined earlier. 
Then $H/I$ is isomorphic to the orthogonal group $O_\varepsilon(2n,2)$ of order $2^{n(n-1)+1}(2^n-\varepsilon)\prod\limits_{i=1}^{n-1}\left(2^{2i}-1\right)$. Here, $\varepsilon=1$ if $P$ is isomorphic to the central product of $n$ dihedral groups of order $8$ and $\varepsilon=-1$ if $P$ is isomorphic to the central product of $n-1$ dihedral group of order $8$ and a quaternion group. 
\end{theorem}
From the above theorem, by selecting appropriate matrix with irreducible characteristic polynomial, it is easy to see that the case $p=2$ follows the exact same pattern as that of $p\neq 2$. So we won't dwell with $p=2$ any further.
\section{MOR cryptosystems on finite $p$-groups using $p$-automorphisms} In the last section we looked at the $p^\prime$-automorphisms. Now we look at the $p$-automorphisms. Our standard reference for $p$-automorphisms is Khukhro ~\cite{khukro}.

To recall, we looked at the exponent-$p$ central series of a $p$-group. It is known that this series has elementary abelian sections. There are two cases with $p$-automorphisms.
\begin{description}
\item[a] The automorphism $\phi$ is not identity on at least one section of the series.
\item[b] The automorphism $\phi$ is identity on all the sections.
\end{description}

In the case a above, one can not build a secure MOR cryptosystem. The reason is as follows:
\begin{theorem}
Let $V$ be a vector space over $\mathbb{F}_q$, a field of characteristic $p>0$. Let $\phi$ be a $p$-automorphism. Then $\phi$ can be written as a block-diagonal matrix with $1$ in the diagonal. Phrased differently, all the eigenvalues of $\phi$ are $1$. 
\end{theorem}
\begin{proof}
The theorem is well-known, see ~\cite[Theorem 2.5]{khukro}.
\end{proof}
Once we have this theorem, the fact that the discrete logarithm problem in that matrix is easy follows from the following observation and the fact that the power of a block diagonal is the power of the respective blocks written as a block diagonal matrix maintaining the order of the block:
\[
\begin{pmatrix}
1	& 1	& \ast & \dots	 & \ast     \\
0	& 1 	& 1    & \dots  & 0 	  \\
\vdots	& 0 	& \ddots&\ddots & \vdots \\
0 	& \dots &\dots & 0	 & 1
\end{pmatrix}^m=\begin{pmatrix}
1	& m	& \ast & \dots	 & \ast     \\
0	& 1 	& m    & \dots  & 0 	  \\
\vdots	& 0 	& \ddots&\ddots & \vdots \\
0 	& \dots &\dots & 0	 & 1
\end{pmatrix}.
\]
This proves that the case a above is useless.

However, the case b above is of immense interest to us. We will give an example of this kind of automorphism. The reason for immense interest is as follows: anyone who is trying to build a new cryptosystem, will want to build a new cryptosystem. In the case of $p^\prime$-automorphisms, in the MOR cryptosystem we saw, the security can be reduced to that of the discrete logarithm problem in matrices. The discrete logarithm problem in matrices is not a new cryptographic primitive. In this case (b above) we have a real good possibility of a new cryptographic primitive.

Let us look at the situation in some details. There are two subgroups of the automorphism group that we are interested in. One is the group of central automorphisms and the other is the group of inner automorphisms.
\subsection{Central automorphisms} Most central automorphisms are $p$-automorphisms. To quote Curran and McCaughan ~\cite{curran}, ``So, roughly speaking, most of the central automorphisms are of $p$-power order''.

Central automorphisms are the centralizer of the group of inner automorphisms in the automorphism group, they form a normal subgroup in the automorphism group. Let $\phi$ be a central automorphism, then $\phi(g)=gz_g$, $z_g\in\mathcal{Z}(G)$. It is clear from the definition that central automorphisms centralize the commutator subgroup. Now take an example of a finite $p$-group $G$, such that $\mathcal{Z}(G)\subseteq G^\prime$. In this group, for a $g\in G$, we have $\phi(g)=gz_g$ and $\phi^m(g)=gz_g^{m}$. So from $g^{-1}\phi(g)$ and $g^{-1}\phi^{m}(g)$, the discrete logarithm problem in the automorphism $\phi$ reduces to the discrete logarithm problem in $z_g\in\mathcal{Z}(G)$. This is exactly the case with the extra-special $p$-group (see Proposition~\ref{u2}). In the case of the extra-special $p$-group of prime exponent, a central automorphisms acts as the identity in both $\mathcal{Z}(G)$ and $G/\mathcal{Z}(G)$. So the obvious way to reduce an automorphism  to matrices over $\mathbb{F}_p$ do not work. However in this case, as demonstrated earlier, it reduces to the discrete logarithm problem in the center. The open question is, can their be other (secure) situations in which the discrete logarithm problem in the automorphism is not the discrete logarithm problem in the usual sense.
\subsection{Inner automorphisms} The group of inner automorphisms of a $p$-group $G$ is a $p$-group. Let $G=G_1\unrhd G_2\unrhd\ldots\unrhd G_k=1$ be a sequence of subgroups in a $p$-group $G$. Let $g\in\text{C}_G(G_2)$ be an element. Then consider the inner automorphism $\phi$ such that $\phi(x)=g^{-1}xg$. Then clearly, $\phi$ acts as the identity on $G_i$ for $i\geq 2$ and $G_i/G_{i+1}$ for $i\geq 1$. However, this is not enough. Recall that our target is, $\phi$ should act like the identity on all possible sections $H/K$ where $\phi$ fixes $K$ and $H/K$ is elementary-abelian. The question is, are there $p$-groups, on which, using the inner automorphisms, one can build a secure MOR cryptosystem?
\section{Conclusion} This paper is a study of finite $p$-groups for the MOR cryptosystem. The aim of this paper was not to provide with a secure MOR cryptosystem. For that, one can look into the arXiv preprint~\cite{UnAyan}. The purpose of this paper is to theoretically justify what can one expect out of finite $p$-groups. There are two classes of automorphisms one should look at. One is $p$-automorphisms and the other is $p^\prime$-automorphisms. The case of $p^\prime$-automorphism has been resolved in this paper as follows: for abelian groups, it is the elementary-abelian $p$-groups. For non-abelian groups, one should use the extra-special $p$-groups of exponent $p$. However there are very interesting questions that are open for $p$-automorphisms. We point those out in this paper.
\bibliography{paper}
\bibliographystyle{amsplain} 
\end{document}